\documentclass[12pt,reqno]{amsart}
\usepackage{amsfonts}
\usepackage{amssymb}
\usepackage{amsmath,amsthm}
\usepackage{bbm}
\usepackage{graphics}
\usepackage{mathrsfs}
\usepackage[cmtip,arrow]{xy}
\usepackage{pb-diagram,pb-xy}

\newcommand{\zz}{\mathbb{Z}}
\newcommand{\nn}{\mathbb{N}}

\newcommand{\qq}{\mathbb{Q}}

\newcommand{\floor}[1]{\left\lfloor#1\right\rfloor}

\newcommand{\isom}{\cong}

\newcommand{\Gal}{\mathrm{Gal}}
\newcommand{\N}{\mathrm{N}}
\newcommand{\p}{\mathfrak{p}}

\newtheorem{lma}{Lemma}

\newtheorem{thm}{Theorem}

\theoremstyle{remark}
\newtheorem{rmk}{Remark}

\usepackage{fullpage}
\begin{document}
\title{A generalization of the Barban-Davenport-Halberstam Theorem to number 
fields\\
(Appeared in \textit{Journal of Number Theory})}

\author{Ethan Smith}
\address{
Department of Mathematical Sciences\\
Clemson University\\
Box 340975 Clemson, SC 29634-0975
}
\email{ethans@math.clemson.edu}
\urladdr{www.math.clemson.edu/~ethans}

\begin{abstract}
For a fixed number field $K$, we consider the mean square error in estimating
the number of primes with norm congruent to $a$ modulo $q$ by the 
Chebotar\"ev Density Theorem when averaging over all $q\le Q$ and all 
appropriate $a$.  
Using a large sieve inequality,
we obtain an upper bound similar to the 
Barban-Davenport-Halberstam Theorem.
\end{abstract}

\maketitle

\section{Introduction}

The mean square error in Dirichlet's Theorem for primes in arithmetic 
progressions was first studied by Barban~\cite{bar:1964} and by Davenport and 
Halberstam~\cite{DH:1966,DH:1968}.
Bounds such as the following are usually referred to as the 
Barban-Davenport-Halberstam Theorem, although this particular refinement 
is attributed to Gallagher.
Let
\begin{equation*}
\psi(x;q,a):=\sum_{\substack{p^m\le x,\\ p^m\equiv a\pmod q}}\log p.
\end{equation*}
Then, for fixed $M>0$,
\begin{equation}\label{bdh}
\sum_{q\le Q}\sum_{\substack{a=1,\\(a,q)=1}}^q
\left(\psi(x;q,a)-\frac{x}{\varphi(q)}\right)^2
\ll xQ\log x
\end{equation}
if $x(\log x)^{-M}\le Q\le x$.
See~\cite[p. 169]{daven:1980}.
Here $\varphi$ is the Euler totient function.  
The sum on the left may be viewed as the mean square error in the 
Chebotar\"ev Density Theorem when averaging over cyclotomic extensions of 
$\qq$.

The inequality in~\eqref{bdh} was later refined by Montgomery~\cite{Mon:1970} 
and Hooley~\cite{Hoo:1975}, both of whom gave asymptotic formulae valid for 
various ranges of $Q$.  See also~\cite[Theorem 1]{Cro:1975}.  
Montgomery's method is based on a result of Lavrik~\cite{Lav:1960} on the 
distribution of twin primes, while Hooley's method relies on the large sieve.
For recent work concerning such asymptotics, see~\cite{liu:2008}.

Results of this type have also been generalized to number fields.  Wilson 
considered error sums over prime ideals falling into a given class of the 
narrow ideal class group in~\cite{wilson:1969}.  While in~\cite{hinz:1981},
Hinz considered sums of principal prime ideals given
by a generator which is congruent to a given algebraic integer modulo an 
integral ideal and whose conjugates fall into a designated range. 

Let $K$ be a fixed algebraic number field which is normal over $\qq$.  In this 
paper, we consider the mean square error in the Chebotar\"ev Density Theorem 
when averaging over cyclotomic extensions of $K$.  That is, we consider sums of 
the form
\begin{equation*}
\psi_K(x;q,a):=\sum_{\substack{\N\mathfrak p^m\le x,\\ \N\mathfrak{p}^m\equiv a\pmod q}}
\log\N\mathfrak{p}.
\end{equation*}
Here the sum is over powers of 
prime ideals of $K$, and there is no restriction to principal primes. 

For each positive integer $q$, we let $A_q:=K\cap\qq(\zeta_q)$. 
So, $A_q$ is an Abelian (possibly trivial) 
extension of $\qq$. 
We have a natural composition of maps:
\begin{equation}\label{comp map}
\begin{diagram}
\node{\Gal(K(\zeta_q)/K)}\arrow{e,J}
\node{\Gal(\qq(\zeta_q)/\qq)}\arrow{e,t}{\sim}
\node{(\zz/q\zz)^*;}
\end{diagram}
\end{equation}
and in fact, 
$\Gal(K(\zeta_q)/K)\isom\Gal(\qq(\zeta_q)/A_q)$.
We let $G_q$ denote the image of composition of maps in~\eqref{comp map}.
Then, in particular, $G_q\isom\Gal(K(\zeta_q)/K)$.  See the diagram below.
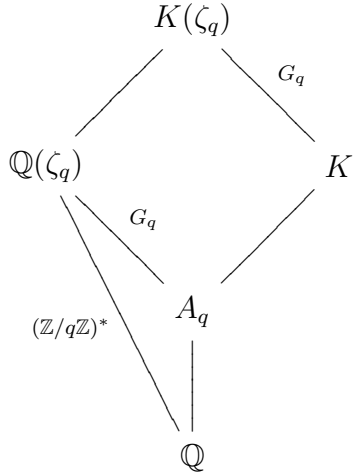
\begin{equation*}
\begin{diagram}
\node{}\node{K(\zeta_q)}\node{}\\
\node{\qq(\zeta_q)}\arrow{ne,-}\node{}\node{K}\arrow{nw,t,-}{G_q}\\
\node{}\node{A_q}\arrow{nw,t,-}{G_q}\arrow{ne,-}\node{}\\
\node{}\node{\qq}\arrow{n,-}\arrow{nnw,b,-}{(\zz/q\zz)^*}\node{}
\end{diagram}
\end{equation*}
Define $\varphi_K(q):=|G_q|$.
By the Chebotar\"ev Density Theorem, for each $a\in G_q$,
\begin{equation*}
\psi_K(x;q,a)=\sum_{\substack{\N\p^m\le x,\\ \N\p^m\equiv a\pmod q}}\log\N\p
\sim\frac{x}{\varphi_K(q)}.
\end{equation*}

\begin{thm}\label{bdh_gen}
For a fixed $M>0$,
\begin{equation*}
\sum_{q\le Q}\sum_{a\in G_q}\left(
\psi_K(x;q,a)-\frac{x}{\varphi_K(q)}
\right)^2
\ll xQ\log x
\end{equation*}
if $x(\log x)^{-M}\le Q\le x$.
\end{thm}
\begin{rmk}
Note that the above agrees with~\eqref{bdh} when $K=\qq$.
The method of proof is essentially an adaptation of 
the proof of~\eqref{bdh} given in~\cite[pp. 169-171]{daven:1980}, the main 
idea being an application of the large sieve.
\end{rmk}
\begin{rmk}
We say nothing about the constant implied by the symbol $\ll$ in the present 
paper.  However, the author has recently adapted the methods of 
Hooley~\cite{Hoo:1975} to refine the result into an asymptotic formula.
This work is the subject of a forthcoming paper.
\end{rmk}
\begin{rmk}
The above result is unconditional and gives a better bound than the Grand 
Riemann Hypothesis (GRH).  See Section~\ref{GRH} for comparison with GRH.
\end{rmk}

\section{Preliminaries and Intermediate Estimates}

We will use lower case Roman letters for rational integers and Fraktur letters
for ideals of the number field $K$.  In particular, $p$ will always denote a 
rational prime and $\mathfrak{p}$ will always denote a prime ideal in 
$\mathcal{O}_K$, the ring of integers of $K$.  We also let
$g(K/\qq ; p)$ and $f(K/\qq ; p)$ denote the number of primes of $K$ lying 
above $p$ and the degree of any prime of $K$ lying above $p$, respectively.
Note that $f(K/\qq; p)$ is well-defined since $K$ is normal over $\qq$.

Let $\mathcal{X}(q)$ denote the character group modulo $q$, 
$\mathcal X^*(q)$ the characters which are primitive modulo
$q$, and let $G_q^\perp$ denote the subgroup of characters that are 
trivial on $G_q$.  Then the character group $\widehat{G_q}$ is 
isomorphic to $\mathcal X(q)/G_q^\perp$, and the number of such characters is 
$\varphi_K(q)=|G_q|=\varphi(q)/|G_q^\perp|$.
As usual, we denote the trivial character of the group 
$\mathcal X(q)$ by $\chi_0$.

For any Hecke character $\xi$ on the ideals of $\mathcal{O}_K$, we define
\begin{equation*}
\psi_K(x,\xi)
:=\sum_{\N\mathfrak{a}\le x}\xi(\mathfrak{a})\Lambda_K(\mathfrak a);
\end{equation*}
and for each character $\chi\in\mathcal X(q)$, we define
\begin{equation*}
\psi_K'(x,\chi\circ\N):=\begin{cases}
\psi_K(x,\chi\circ\N),& \chi\not\equiv\chi_0\pmod{G_q^\perp},\\
\psi_K(x,\chi\circ\N)-x,& \chi\equiv\chi_0\pmod{G_q^\perp}.
\end{cases}
\end{equation*}
Here, $\Lambda_K$ is the von Mangoldt function defined on the ideals of 
$\mathcal{O}_K$, i.e.,
\begin{equation*}
\Lambda_K(\mathfrak a):=\begin{cases}
\log\N\mathfrak p,& \mathfrak a=\mathfrak p^m,\\
0,&\mathrm{otherwise}.
\end{cases}
\end{equation*}

\begin{lma}\label{lsieve_app}
\begin{equation*}
\sum_{q\le Q}\frac{q}{\varphi(q)}\sum_{\chi\in\mathcal X^*(q)}|\psi_K(x,\chi\circ\N)|^2
\ll (x+Q^2)x\log x.
\end{equation*}
\end{lma}
\begin{proof}
For $n\in\nn$, we define
\begin{eqnarray*}
D_K(n)&:=&\#\{\mathfrak{p}^m\lhd\mathcal{O}_K:\N\mathfrak{p}^m =n\};\\
\Lambda_K^*(n)&:=&\begin{cases}
\log p^{f(K/\qq;p)},& n=p^k,\\
0,&\mathrm{otherwise}.
\end{cases}
\end{eqnarray*} 
Now, note that
\begin{eqnarray*}
\psi_K(x,\chi\circ\N)
=\sum_{\N\mathfrak{a}\le x}\chi(\N\mathfrak a)\Lambda_K(\mathfrak a)
=\sum_{n\le x}\chi(n)D_K(n)\Lambda_K^*(n).
\end{eqnarray*}
We apply the large sieve in the form 
of Theorem 4 in chapter 27 of~\cite{daven:1980} to see that
\begin{eqnarray*}
\sum_{q\le Q}\frac{q}{\varphi(q)}\sum_{\chi\in\mathcal X^*(q)}|\psi_K(x,\chi\circ\N)|^2
&\ll&(x+Q^2)\sum_{n\le x}\left(D_K(n)\Lambda_K^*(n)\right)^2\\
&=&(x+Q^2)\sum_{p^k\le x}
g(K/\qq; p)
D_K(p^k)\Lambda_K^*(p^k)^2\\
&\ll&(x+Q^2)
\sum_{\N\mathfrak p^m\le x}\Lambda_K(\mathfrak p^m)^2\\
&\ll&(x+Q^2)x\log x.
\end{eqnarray*}
\end{proof}

\begin{lma}\label{char_exchange}
If $\xi_1$ and $\xi_2$ are Hecke characters modulo $\mathfrak{q}_1$ and 
$\mathfrak{q}_2$ respectively, and if $\xi_1$ induces $\xi_2$, then 
\begin{equation*}
\psi_K(x,\xi_2)=\psi_K(x,\xi_1)+O\left((\log qx)^{2}\right),
\end{equation*}
where $(q)=\mathfrak{q}_2\cap\zz$.
\end{lma}
\begin{proof}
\begin{alignat*}{2}
|\psi_K(x,\xi_1)-\psi_K(x,\xi_2)| &=
\left|\sum_{\substack{\N\mathfrak{p}^m\le x,\\(\mathfrak{p},\mathfrak{q}_2)>1}}
\xi_1(\mathfrak{p}^m)\log\N\mathfrak{p}\right|\nonumber
\le\sum_{\substack{p^k\le x,\\ (p,q)>1}}D_K(p^k)f(K/\qq;p)\log p\nonumber\\
&=\sum_{p|q}\sum_{\substack{k=1,\\ f(K/\qq;p)|k}}^{\floor{\log x/\log p}}
g(K/\qq; p)f(K/\qq; p)\log p
\nonumber\\
&\ll\sum_{p|q}\floor{\frac{\log x}{\log p}}\log p
\ll(\log qx)^{2}.
\end{alignat*}
\end{proof}

\begin{lma}\label{psi_bound}
If $\chi$ is a character modulo $q\le(\log x)^{M+1}$, then there exists a 
positive constant $C$ (depending on $M$) such that
\begin{equation*}
\psi_K'(x,\chi\circ\N)\ll x\exp\left\{-C\sqrt{\log x}\right\}.
\end{equation*}
\end{lma}
\begin{proof}
As a Hecke character on the ideals of $\mathcal{O}_K$, 
$\chi\circ\N$ may not be primitive modulo $q\mathcal{O}_K$.  
Let $\xi=\xi_\chi$ be the primitive Hecke character which induces 
$\chi\circ\N$,
and let $\mathfrak{f}_\chi$ be its conductor.
Write $s=\sigma+it$.
By~\cite[Theorem 5.35]{IK:2004}, there exists an effective constant 
$c_0>0$ such that the Hecke $L$-function
$L(s,\xi):=\sum_{\N\mathfrak a\le x}\xi(\mathfrak a)(\N\mathfrak a)^{-s}$ 
has at most one zero in the region
\begin{equation}\label{zero-free region}
\sigma>1-\frac{c_0}{[K:\qq]\log\left(|d_K|\N\mathfrak f(|t|+3)\right)},
\end{equation}
where $d_K$ denotes the discriminant of the number field.
Further, if such a zero exists, it is real and simple.
In the case that such a zero exists, we call it an ``exceptional zero" and 
denote it by $\beta_\xi$.
Thus, by~\cite[Theorem 5.13]{IK:2004}, 
there exists $c_1>0$ such that 
\begin{equation*}
\psi_K(x,\xi)
=\delta_\xi x-\frac{x^{\beta_\xi}}{\beta_\xi}
+O\left(x\exp\left\{
\frac{-c_1\log x}{\sqrt{\log x}+\log\N\mathfrak{f}_\chi}
\right\}(\log(x\N\mathfrak{f}_\chi))^4\right),
\end{equation*}
where 
\begin{equation*}
\delta_\xi=\begin{cases}
1, & \xi\mbox{ trivial},\\
0, & \mbox{otherwise},
\end{cases}
\end{equation*}
and the term $x^{\beta_\xi}/\beta_\xi$ is omitted if the $L$-function 
$L(s,\xi)$
has no exceptional zero in the region~\eqref{zero-free region}.  
Now, since $\mathfrak{f}_\chi|q\mathcal{O}_K$ and $q\le(\log x)^{M+1}$, we 
have the following bound on the error term: 
\begin{equation*}
x\exp\left\{
\frac{-c_1\log x}{\sqrt{\log x}+\log\N\mathfrak{f}_\chi}
\right\}(\log(x\N\mathfrak{f}_\chi))^4
\ll x\exp\left\{-c_2\sqrt{\log x}\right\}
\end{equation*}
for some positive constant $c_2$.

By~\cite[Theorem 3.3.2]{gold:1970}, we see that
for every $\epsilon>0$, 
there exists a constant $c_\epsilon>0$ such that 
if $\beta_\xi$ is an exceptional zero for $L(s,\xi)$, then
\begin{equation*}
\beta_\xi < 1-\frac{c_\epsilon}{(\N\mathfrak{f}_\chi)^\epsilon}
\le 1-\frac{c_\epsilon}{q^{[K:\qq]\epsilon}}.
\end{equation*}
Thus, 
\begin{equation*}
x^{\beta_\xi}<x\exp\left\{-c_\epsilon(\log x)q^{-[K:\qq]\epsilon}\right\}
<x\exp\left\{-c_\epsilon(\log x)^{1/2}\right\}
\end{equation*}
upon choosing $\epsilon$ so that $[K:\qq]\epsilon=(2M+2)^{-1}$.
Whence, for $q\le(\log x)^{M+1}$, there exists $C>0$ such that
\begin{equation*}
\psi_K(x,\xi)
=\delta_\xi x
+O\left(x\exp\left\{-C\sqrt{\log x}\right\}\right).
\end{equation*}
Therefore, by Lemma~\ref{char_exchange},
\begin{equation*}
\psi'_K(x,\chi\circ\N)\ll x\exp\left\{-C\sqrt{\log x}\right\}
\end{equation*}
for $q\le (\log x)^{M+1}$.
\end{proof}

\section{Proof of Main Theorem}

For $a\in G_q$, we define the error term
$E_K(x;q,a):=\psi_K(x;q,a)-\frac{x}{\varphi_K(q)}$,
and note that
\begin{equation*}
E_K(x;q,a)
=\frac{1}{\varphi_K(q)}
\sum_{\chi\in\widehat{G_q}}\bar{\chi}(a)\psi'_K(x,\chi\circ\N).
\end{equation*}
Now we form the square of the Euclidean norm and sum over all 
$a\in G_q$ to see
\begin{eqnarray*}
\sum_{a\in G_q}\left|E_K(x;q,a)\right|^2&=&
\frac{1}{\varphi_K(q)^2}
\sum_{a\in G_q}
\left|\sum_{\chi\in\widehat{G_q}}\bar\chi(a)\psi'_K(x,\chi\circ\N)\right|^2\\
&=&
\frac{1}{\varphi_K(q)^2}
\sum_{a\in G_q}
\sum_{\chi_1\in\widehat{G_q}}
\sum_{\chi_2\in\widehat{G_q}}
\bar\chi_1(a)\chi_2(a)
\psi'_K(x,\chi_1\circ\N)\overline{\psi'_K(x,\chi_2\circ\N)}\\
&=&
\frac{1}{\varphi_K(q)}
\sum_{\chi\in\widehat{G_q}}\left|\psi'_K(x,\chi\circ\N)\right|^2\\
&=&\frac{1}{\varphi(q)}\sum_{\chi\in\mathcal X(q)}\left|\psi'_K(x,\chi\circ\N)\right|^2.
\end{eqnarray*}
For each $\chi\in\mathcal X(q)$, we let $\chi_*$ denote the primitive character which 
induces $\chi$.  By Lemma~\ref{char_exchange}, we have
$\psi_K'(x,\chi\circ\N)
=\psi_K'(x,\chi_*\circ\N)+O\left((\log qx)^{2}\right)$.
Hence, summing over $q\le Q$ and exchanging each character for its primitive 
version, we have
\begin{equation*}
\sum_{q\le Q}\sum_{a\in G_q}E_K(x;q,a)^2
\ll\sum_{q\le Q}(\log qx)^{4}
+\sum_{q\le Q}\frac{1}{\varphi(q)}\sum_{\chi\in\mathcal X(q)}
|\psi_K'(x,\chi_*\circ\N)|^2.
\end{equation*}
The first term on the right is clearly smaller than $xQ\log x$, so 
we concentrate on the second.
Now,
\begin{eqnarray}
\sum_{q\le Q}\frac{1}{\varphi(q)}\sum_{\chi\in\mathcal X(q)}|\psi_K'(x,\chi_*\circ\N)|^2
&=&\sum_{q\le Q}\sum_{\chi\in\mathcal X^*(q)}|\psi_K'(x,\chi\circ\N)|^2
\sum_{k\le Q/q}\frac{1}{\varphi(kq)}\nonumber \\
&\ll&\sum_{q\le Q}\frac{1}{\varphi(q)}\left(\log\frac{2Q}{q}\right)
\sum_{\chi\in\mathcal X^*(q)}|\psi_K'(x,\chi\circ\N)|^2\label{eqn}
\end{eqnarray}
since $\sum_{k\le Q/q}1/\varphi(kq)\ll\varphi(q)^{-1}\log(2Q/q)$.
See~\cite[p. 170]{daven:1980}.  The proof will be complete once we show
that~\eqref{eqn} is smaller than $xQ\log x$ for $Q$ in the 
specified range.

As with the proof of~\eqref{bdh} in~\cite[pp. 169-171]{daven:1980}, we consider large and small $q$ separately.  We start with the large values.  Since 
$\psi_K'(x,\chi\circ\N)\ll\psi_K(x,\chi\circ\N)$, by
Lemma~\ref{lsieve_app}, we have
\begin{eqnarray*}
\sum_{U<q\le 2U}\frac{U}{\varphi(q)}\sum_{\chi\in\mathcal X^*(q)}
|\psi'_K(x,\chi\circ\N)|^2
&\ll&(x+U^2)x\log x,
\end{eqnarray*}
which implies
\begin{equation*}
\sum_{U\le q\le 2U}\frac{1}{\varphi(q)}\left(\log\frac{2Q}{q}\right)
\sum_{\chi\in\mathcal X^*(q)}|\psi_K'(x,\chi\circ\N)|^2
\ll(xU^{-1}+U)x\log x\left(\log\frac{2Q}{U}\right)
\end{equation*}
for $1\le 2U\le Q$.  Summing over $U=Q2^{-k}$, we have
\begin{eqnarray}\label{largeq}
\sum_{Q_1<q\le Q}\frac{1}{\varphi(q)}\left(\log\frac{2Q}{q}\right)
\sum_{\chi\in\mathcal X^*(q)}|\psi_K'(x,\chi\circ\N)|^2
&\ll&x\log x
\sum_{k=1}^{\floor{\frac{\log(Q/Q_1)}{\log 2}}}(x2^kQ^{-1}+Q2^{-k})\nonumber\\
&\ll&
x^2Q_1^{-1}(\log x)\log Q+xQ\log x\nonumber\\
&\ll& xQ\log x
\end{eqnarray}
if $x(\log x)^{-M}\le Q\le x$ and $Q_1=(\log x)^{M+1}$.

We now turn to the small values of $q$.  Applying Lemma~\ref{psi_bound}, we have
\begin{eqnarray}\label{smallq}
\sum_{q\le Q_1}\frac{1}{\varphi(q)}\left(\log\frac{2Q}{q}\right)
\sum_{\chi\in\mathcal X^*(q)}|\psi_K'(x,\chi\circ\N)|^2
&\ll&Q_1(\log Q)\left(x\exp\left\{-c\sqrt{\log x}\right\}\right)^2
\nonumber\\
&\ll&x^2(\log x)^{-M}
\ll xQ\log x.
\end{eqnarray}
Combining~\eqref{largeq} and~\eqref{smallq}, the theorem follows.
\hfill $\Box$

\section{Comparison with GRH}\label{GRH}

Using the bound on the analytic conductor of the $L$-function 
$L(s,\chi\circ\N)$ given in~\cite[p. 129]{IK:2004}, GRH implies
\begin{eqnarray*}
\sum_{q\le Q}\sum_{a\in G_q}E_K(x;q,a)^2
&=&\sum_{q\le Q}\frac{1}{\varphi(q)}
\sum_{\chi\in\mathcal X(q)}\left|\psi'_K(x,\chi\circ\N)\right|^2\\
&\ll&(\sqrt x(\log x)^2)^2\sum_{q\le Q}\frac{1}{\varphi(q)}
\sum_{\chi\in\mathcal X(q)}1\\
&=& xQ(\log x)^4.
\end{eqnarray*}
See~\cite[Theorem 5.15]{IK:2004} for this implication of GRH.

\section{Acknowledgment}
The main theorem of this paper is an estimate required by the author's work in 
his PhD dissertation.  The author is grateful for the guidance of his advisor, 
Kevin James.

\bibliographystyle{plain}
\bibliography{references}

\def\cprime{$'$}
\begin{thebibliography}{10}

\bibitem{bar:1964}
M.B. Barban.
\newblock On the distribution of primes in arithmetic progressions ``on
  average".
\newblock {\em Dokl. Akad. Nauk UzSSR}, 5:5--7, 1964.
\newblock (Russian).

\bibitem{Cro:1975}
M.~J. Croft.
\newblock Square-free numbers in arithmetic progressions.
\newblock {\em Proc. London Math. Soc. (3)}, 30:143--159, 1975.

\bibitem{DH:1966}
H.~Davenport and H.~Halberstam.
\newblock Primes in arithmetic progressions.
\newblock {\em Michigan Math. J.}, 13:485--489, 1966.

\bibitem{DH:1968}
H.~Davenport and H.~Halberstam.
\newblock Corrigendum: ``{P}rimes in arithmetic progression''.
\newblock {\em Michigan Math. J.}, 15:505, 1968.

\bibitem{daven:1980}
Harold Davenport.
\newblock {\em Multiplicative Number Theory}.
\newblock Springer-Verlag, New York, 1980.

\bibitem{gold:1970}
Larry~Joel Goldstein.
\newblock A generalization of the {S}iegel-{W}alfisz theorem.
\newblock {\em Trans. Amer. Math. Soc.}, 149:417--429, 1970.

\bibitem{hinz:1981}
J{\"u}rgen~G. Hinz.
\newblock On the theorem of {B}arban and {D}avenport-{H}alberstam in algebraic
  number fields.
\newblock {\em J. Number Theory}, 13(4):463--484, 1981.

\bibitem{Hoo:1975}
Christopher Hooley.
\newblock On the {B}arban-{D}avenport-{H}alberstam theorem. {I}.
\newblock {\em J. Reine Angew. Math.}, 274/275:206--223, 1975.
\newblock Collection of articles dedicated to Helmut Hasse on his seventy-fifth
  birthday, III.

\bibitem{IK:2004}
Henryk Iwaniec and Emmanuel Kowalski.
\newblock {\em Analytic Number Theory}, volume~53 of {\em Colloquium
  Publications}.
\newblock American Mathematical Society, Providence, 2004.

\bibitem{Lav:1960}
A.~F. Lavrik.
\newblock On the twin prime hypothesis of the theory of primes by the method of
  {I}. {M}. {V}inogradov.
\newblock {\em Soviet Math. Dokl.}, 1:700--702, 1960.

\bibitem{liu:2008}
H.-Q. Liu.
\newblock Barban-{D}avenport-{H}alberstam average sum and exceptional zero of
  ${L}$-functions.
\newblock {\em J. Number Theory}, 121(4):1044--1059, 2008.

\bibitem{Mon:1970}
H.~L. Montgomery.
\newblock Primes in arithmetic progressions.
\newblock {\em Michigan Math. J.}, 17:33--39, 1970.

\bibitem{wilson:1969}
Robin~J. Wilson.
\newblock The large sieve in algebraic number fields.
\newblock {\em Mathematika}, 16:189--204, 1969.

\end{thebibliography}
\end{document}